\documentclass[10pt]{amsart}
\usepackage{mathrsfs}
\renewcommand{\mathcal}{\mathscr}
\usepackage{amsmath,amsthm,amsopn,amstext,amscd,amsfonts,amssymb}
\usepackage{graphicx, color, enumerate}
\usepackage[latin1]{inputenc}
\usepackage[active]{srcltx}
\usepackage{tikz}
\usepackage{pgf}
\usepackage{etex}
\usepackage{verbatim}
\usepackage{tikz-3dplot}
\usepackage{pgfkeys}
\usepackage{fp}

\numberwithin{equation}{section}
\newcounter{bbb}
\numberwithin{bbb}{section}

\newtheorem{theorem}[bbb]{Theorem}
\newtheorem{lemma}[bbb]{Lemma}
\newtheorem{proposition}[bbb]{Proposition}
\newtheorem{corollary}[bbb]{Corollary}
\newtheorem{definition}[bbb]{Definition}
\newtheorem{remark}[bbb]{Remark}

\renewcommand{\geq}{\geqslant}
\renewcommand{\ge}{\geqslant}
\renewcommand{\leq}{\leqslant}
\renewcommand{\le}{\leqslant}

\usepackage{color}

\begin{document}

\title[The heat equation with nonlocal diffusion]{A Widder's type Theorem for the heat equation with nonlocal diffusion}

\author{ Bego\~na Barrios, Ireneo Peral, Fernando Soria, Enrico Valdinoci}
\thanks{Work partially supported by project
MTM2010-18128, MICINN.}
\address{ Bego\~na Barrios, Ireneo Peral, Fernando Soria,
\hfill\break\indent Departamento de Matem{\'a}ticas, Universidad Autonoma
de Madrid, Spain.
\hfill\break\indent Enrico Valdinoci,
\hfill\break\indent Dipartimento di Matematica, Universit\`a degli Studi di Milano, Italy.}

\email{{\tt bego.barrios@uam.es, ireneo.peral@uam.es,\hfill\break\indent fernando.soria@uam.es, enrico.valdinoci@unimi.it}}

\begin{abstract}
The main goal of this work is to prove that every non-negative
{\it strong solution} $u(x,t)$  to the problem
$$
    u_t+(-\Delta)^{\alpha/2}u=0 \ \quad\mbox{for }
    (x,t)\in\mathbb{R}^{n}\times(0,T), \quad 0<\alpha<2,
$$
 can be written as
$$u(x,t)=\int_{\mathbb{R}^{n}}{P_{t}(x-y)u(y,0)\, dy},$$
where
$$P_{t}(x)=\frac{1}{t^{n/\alpha}}P\left(\frac{x}{t^{1/\alpha}}\right),
$$
and
$$
P(x):=\int_{\mathbb{R}^{n}}{e^{ix\cdot\xi-|\xi|^{\alpha}}d\xi}.
$$
This result shows uniqueness in the setting of non-negative
solutions and extends some classical
results for the heat equation by D. V. Widder in \cite{W0} to the
nonlocal diffusion framework.
\end{abstract}

\date{\today}
\maketitle

\section{Introduction}
The heat equation  has the hyperplane $t=0$ as a characteristic surface and this causes that the Cauchy problem with initial data on $t=0$
is not well possed  in general. However there is a clear agreement between the
Principles of Thermodynamics and the model of
transfer of heat given by such an equation. This is reflected into the fact
that the initial temperature evolves in time as the convolution with a kernel,
giving rise to an average, and smoothing out in this way the potential
effect of sharp thermal differences (this is in agreement with
the entropy effect of the Second Principle of Thermodynamics);
also uniqueness holds true
under a positivity assumption on the function (this is in agreement
with the so called Third Principle of Thermodynamics, according to which
temperatures are always positive if measured in
the Kelvin scale). In this sense D. V. Widder in \cite{W0} proved the following classical result:

\

{\it Assume that $u:\mathbb{R}^n\times [0,T)\subset\overline{\mathbb{R}^n_+}\rightarrow\mathbb{R}$ is so that
$$u(x,t)\ge 0, \quad u\in \mathcal{C}(\mathbb{R}^n\times [0,T)), \quad u_t, u_{x_i x_i}\in \mathcal{C}(\mathbb{R}^n\times (0,T)),$$
and  satisfies
$$u_t(x,t)-\Delta u(x,t)=0,\quad (x,t)\in \mathbb{R}^n\times (0,T),$$
in the classical sense.

Then
$$u(x,t)=\frac{1}{(4\pi t)^{\frac n2}}\int_{\mathbb{R}^n} u(y,0) \exp\Big(\frac{-|x-y|^2}{4t}\Big)d y.$$
}

\

In this paper we obtain a similar result for the {\it nonlocal heat equation}, that is the equation,
\begin{equation}\label{eq:base}
    u_t+(-\Delta)^{\alpha/2}u=0 \ \quad\mbox{ for }     (x,t)\in\mathbb{R}^{n}\times(0,T), \quad 0<\alpha<2,
\end{equation}
in which the diffusion is given by a power of the Laplacian (we refer to \cite{npv} and \cite{stein} for basic definitions and properties
of the fractional Laplace operator).

It is worthy to point out that for every $\alpha\in(0,2]$ the operator in~\eqref{eq:base} does not satisfies the so called {\it Hadamard condition},
that is, the Cauchy problem is ill possed (see \cite{had} for more details). As a consequence we face here a problem similar to the one of the heat equation.

It is well known that the operator in~\eqref{eq:base}
and its inverse are nonlocal.  In a sense, equation~\eqref{eq:base} is perhaps one of the simplest
classical examples of pseudodifferential operators.
Notice that if we consider the problem
$$
    u_t+(-\Delta)^{\alpha/2}u=0 \qquad\mbox{for }
    (x,t)\in\mathbb{R}^{n}\times(0,T), \quad 0<\alpha<2$$
with initial datum $u(x,0)=u_0(x)$
assuming, say, $u_0\in \mathcal{C}(\mathbb{R}^n)\bigcap L^\infty(\mathbb{R}^n)$, then a solution is obtained by the convolution with the kernel
\begin{equation}\label{nucleot}
P_{t}(x)=\frac{1}{t^{n/\alpha}}P\left(\frac{x}{t^{1/\alpha}}\right),
\end{equation}
where
\begin{equation}\label{1.2BIS}
P(x):=\int_{\mathbb{R}^{n}}{e^{ix\cdot\xi-|\xi|^{\alpha}}d\xi}.
\end{equation}
That is, a solution is given by
\begin{equation}\label{convolution}
u(x,t)=\int_{\mathbb{R}^{n}}{P_{t}(x-y)u_0(y)\, dy}.
\end{equation}
We look here for a class of solutions of the fractional parabolic equation, such that the sign condition $u(x,t)\ge 0$ imply that  $u$,  necessarily,
is of the form \eqref{convolution} with $u_0(x)$  replaced by the trace $u(x,0)$.
This is the type of extension that we propose
for the classical result of Widder to the nonlocal equation in~\eqref{eq:base}.

In this article we will describe several interpretations of what {\it solution to equation}  \eqref{eq:base} means, and consider weak, viscosity and strong solutions, in a sense that we now make more precise.

\

\subsection*{Weak solutions} Consider the space
$$\mathcal{L}^{\alpha/2}(\mathbb{R}^{n}):=\left\{
u:\mathbb{R}^{n}\to\mathbb{R}\quad \mbox{measurable}:\, \int_{\mathbb{R}^{n}}{\frac{|u(x)|}{1+|x|^{n+\alpha}}\, dx}<\infty\right\},$$
endowed with the norm
$$\|u\|_{\mathcal{L}^{\alpha/2}(\mathbb{R}^{n})}:=\int_{\mathbb{R}^{n}}{
\frac{|u(x)|}{ 1+|x|^{n+\alpha}}\, dx}.$$
If $u\in\mathcal{L}^{\alpha/2}(\mathbb{R}^{n})$ then $(-\Delta)^{\alpha/2}u$ can be defined in the weak sense as a distribution, that is we can compute the duality product $\langle (-\Delta)^{\alpha/2}u, \varphi\rangle$, for every $\varphi$ in the Schwartz class.

\begin{definition}
We say that $u(x,t)$ is a weak solution of the fractional heat problem
$$
\left\{\begin{array}{ll}
    u_t+(-\Delta)^{\alpha/2}u=0 &\quad\mbox{for }
    (x,t)\in\mathbb{R}^{n}\times(0,T), \\
    u(x,0)=u_0(x)&\quad\mbox{in } \mathbb{R}^{n},
  \end{array}\right.
$$
if  the following conditions hold:
\begin{itemize}
\item[i)]{$u\in L^{1}([0,T'],\mathcal{L}^{\alpha/2}(\mathbb{R}^{n}))$ for every $T'<T$.}
\item[ii)]{$u\in \mathcal{C} ((0,T), {L}^{1}_{loc}(\mathbb{R}^{n}))$.}
\item[iii)]{For every test function $\varphi\in C^{\infty}_{0}(\mathbb{R}^{n}\times[0,T))$ and $0<T'<T$ one has that
\begin{equation}\int_{0}^{T'}{\int_{\mathbb{R}^{n}}{\left[-u(x,t)\varphi_t(x,t)+u(x,t)(-\Delta)^{\alpha/2}\varphi(x,t)\right]\, dx\, dt}}=
\label{WS} \end{equation}
$$\int_{\mathbb{R}^{n}}{u(x,T')\varphi(x,T')\, dx}-\int_{\mathbb{R}^{n}}{u_0(x)\varphi(x,0)\, dx}.
$$}
\end{itemize}
\end{definition}
Condition ii) is imposed so that the right hand side of equality (\ref{WS}) makes sense for every $T'$. Notice that the continuity  would follow in any case from the left hand side of (\ref{WS}) due to the integrability condition in i).

\

\subsection*{Viscosity solutions} Consider $Q=\mathbb{R}^n\times(0,T)$ and
\begin{eqnarray*}
\mathcal{C}^{1,2}_p(Q)&=&\Big\{
f:Q\rightarrow \mathbb{R}\,
|\qquad f_t\in \mathcal{C}(Q),
\quad f_{x_i,x_j}\in \mathcal{C}(Q)\\
&&{\mbox{and }}
\sup_{t\in(0,T)}|f(x,t)|\le C(1+|x|)^p
\Big\}\end{eqnarray*}
\begin{definition}\label{def.visc.parabolic}
A function $u\in\mathcal{C}(Q)$ is a
\emph{viscosity subsolution (resp. supersolution)} of
\begin{equation}\label{E:P}
u_t+(-\Delta)^{\alpha/2}\big(u\big)= 0
\end{equation}
in $Q$ if for all $(\hat{x},\hat{t})\in Q$ and  $\varphi\in \mathcal{C}^{1,2}_p(Q)$
such that $u-\varphi$ attains a local maximum (minimum) at  $(\hat{x},\hat{t})$ one has
\[
\varphi_t(\hat x,\hat t)+(-\Delta)^{\alpha/2}\big(\varphi(\hat{x},\hat t))\big)\leq 0\qquad(\text{resp.}\ \geq).
\]

We say that $u\in\mathcal{C}(Q)$ is a \emph{viscosity solution} of
\eqref{E:P} in $Q$ if it is both a viscosity subsolution and supersolution. See \cite{Imbert} for more details about these type of solutions.
\end{definition}

\

\subsection*{Strong solutions}
We know (see Proposition 2.1.4 of \cite{S}) that if $u\in \mathcal{L}^{\alpha/2}(\mathbb{R}^{n})\cap \mathcal{C}^{\alpha+\gamma}_{\rm loc}(\mathbb{R}^{n})$ (or $\mathcal{C}^{1,\alpha+\gamma-1}$ if $\alpha>1$), for some $\gamma>0$ then $(-\Delta)^{\alpha/2}u$ is a continuous function and it is well defined as the following principal value
$$(-\Delta)^{\alpha/2}u(x):=C(n,\alpha)\mbox{P.V.}\int_{\mathbb{R}^{n}}{\frac{u(x)-u(y)}{|x-y|^{n+\alpha}}\, dy},\, \alpha\in(0,2).$$
Here $C(n,\alpha)$ denotes the constant satisfying the identity
$$(-\Delta)^{\alpha/2}u=\mathcal{F}^{-1}(|\xi|^{\alpha}\mathcal{F}u),\, \xi\in\mathbb{R}^{n},\, u\in\mathcal{S},\, \alpha\in(0,2);$$
that is,
$$C(n,\alpha)=\left(\int_{\mathbb{R}^{n}}{\frac{1-cos(\xi_1)}{|\xi|^{n+\alpha}}d\xi}\right)^{-1},$$
see \cite{npv}.

\

In order to allow a larger class of functions, and following standard procedures in the theory of singular integrals, we will define hereafter  the principal value as the two-sided limit
$$\mbox{P.V.}\int_{\mathbb{R}^{n}}{\frac{u(x)-u(y)}{|x-y|^{n+\alpha}}\, dy}=\lim_{\epsilon\to 0}\int\limits_{\{y\,|\, \epsilon<|x-y|< 1/\epsilon\}}{\frac{u(x)-u(y)}{|x-y|^{n+\alpha}}\, dy}.$$
We  will use the same notation for this extended operator. When $u\in \mathcal{L}^{\alpha/2}(\mathbb R^n)$ this definition coincides with the usual one.

\begin{definition}\label{strong}
We say that $u(x,t)$ is a strong solution of the fractional heat equation
$$
u_t+(-\Delta)^{\alpha/2}u=0 \quad\mbox{for }(x,t)\in\mathbb{R}^{n}\times (0,T),
$$
if the following conditions hold:
\begin{itemize}
\item[i)]{$u_t(x,t)\in \mathcal{C}(\mathbb{R}^{n}\times(0,T))$.}
\item[ii)]{$u\in \mathcal{C}(\mathbb{R}^{n}\times[0,T))$.}
\item[iii)]{The equation is satisfied pointwise for every $(x,t)\in\mathbb{R}^{n}\times(0,T)$, that is,
$$u_t(x,t)+C(n,\alpha)\mbox{{\rm P.V.}}\int_{\mathbb{R}^{n}}{\frac{u(x,t)-u(y,t)}{|x-y|^{n+\alpha}}\, dy}=0.$$
}
\end{itemize}
\end{definition}\
%%%%%%%%%%%%%%%%%
Note that if $u(x,t)$ is a strong solution then
$$\mbox{P.V.}\int_{\mathbb{R}^{n}}{\frac{u(x,t)-u(y,t)}{|x-y|^{n+\alpha}}\, dy}\in \mathcal{C}(\mathbb{R}^{n}\times(0,T)).$$
%%%%%%%%%%%%%%%%%%

Observe also that if $u(x,t)$ is a strong solution of the fractional heat equation satisfying $u(x,t)\in L^{1}([0,T'],\mathcal{L}^{\alpha/2}(\mathbb{R}^{n}))$ for every $T'<T$, then $u(x,t)$ is a weak solution of the fractional heat equation
(as a byproduct of our results, we will see that this holds true for non-negative
strong solutions, see the forthcoming Corollary~\ref{corol4}).

\

Our main contribution in this setting is the following Widder's type result

\begin{theorem}\label{MAIN}
If $u\ge0$ is a strong solution of
$$
    u_t+(-\Delta)^{\alpha/2}u=0 \ \quad\mbox{for }
    (x,t)\in\mathbb{R}^{n}\times(0,T), \quad 0<\alpha<2,
$$
then
$$u(x,t)=\int_{\mathbb{R}^{n}}{P_{t}(x-y)u(y,0)\, dy}.$$
\end{theorem}

Note that, if for a given polynomial $p$ with real coefficients we consider the differential operator $p(D)$, then we obtain that
$$p(D)P_{t}(x)=\int_{\mathbb{R}^{n}}{e^{ix\cdot\xi}p(i\xi)e^{-t|\xi|^{\alpha}}\, d\xi}.$$
Since $e^{-t|\xi|^{\alpha}}$ is a tempered distribution we deduce that $P_{t}\in C^{\infty}(\mathbb{R}^{n}\times(0,\infty))$ (see for instance \cite{droniou} and \cite{B1} for more details).
%and, moreover,
%$$\frac{1}{C_{m}}\frac{1}{1+|x|^{n+\alpha+m}}\leq \partial_{x}^{m}P(x)\leq \frac{C_{m}}{1+|x|^{n+\alpha+m}}.$$
In particular, and as a consequence of Theorem \ref{MAIN}, we get  that if $u$ is a non negative strong solution of the fractional heat equation then $u\in C^{\infty}(\mathbb{R}^{n}\times(0,T))$.

The paper is organized as follows.
In Section \ref{uniweak} we prove a uniqueness result for weak solutions that, in turn, will be the key step to obtain our representation theorem.
Next in Section \ref{postrong} we prove the main result for strong solutions.
We start by proving a comparison result that will allow us to show that every positive strong solution $u(x,t)$ is bigger than or equal to  the convolution of the
trace $u(x,0)$ with the kernel $P_t(x)$. By a scaling argument we will prove  that any positive strong solution is also a weak solution and then, by the uniqueness result of the previous section, we will conclude with the proof of the theorem.

Finally, in the last section  we establish the pointwise behavior of positive strong solution that has an interest on its own and that could provide and alternative proof to our representation result.

\section{Uniqueness for weak solutions}\label{uniweak}

To start with, we prove a uniqueness result for weak solutions with
vanishing initial condition. The proof is quite complicated
and it involves many fine integral estimates. The nonlocal feature
of the problem also makes localization and cutoff arguments much harder
than in the classical case.

\begin{theorem}\label{Th1}
Set $T>0$, $\alpha\in(0,2)$ and let $u$ be a weak solution of the fractional heat equation
\begin{equation}\label{2.00}
\left\{\begin{array}{ll}
    u_t+(-\Delta)^{\alpha/2}u=0 &\quad\mbox{for }
    (x,t)\in\mathbb{R}^{n}\times(0,T), \\
    u(x,0)=0&\quad\mbox{in } \mathbb{R}^{n}.
  \end{array}\right.
\end{equation}
Then $u(x,t)=0$ for every $t\in(0,T)$ and a.e. $x\in\mathbb{R}^{n}$.
\end{theorem}
%%%%%%%%%%%%%
\begin{proof}
We must show that $u(x,t_0)=0$ for an arbitrary $t_0\in (0,T)$ and $x\in \mathbb R^n$. For this,
we fix $R_0>0$ and $\theta\in \mathcal{C}_{0}^{\infty}(B_{R_0})$ and we will prove that
$$\int_{\mathbb{R}^{n}}{u(x,t_0)\theta(x)\, dx}=0.$$
For any $t\in [0,t_0)$, we define
$$\varphi(x,t):=(\theta(\cdot)\ast P_{t_{0}-t}(\cdot))(x),$$
where $P_t$ is the kernel defined in \eqref{nucleot} and \eqref{1.2BIS}.
By \cite{K} (see also \cite{B} and \cite{CR2}) we know that
\begin{equation}\label{P}
\frac{1}{C}\frac{1}{1+|x|^{n+\alpha}}\leq P(x)\leq \frac{C}{1+|x|^{n+\alpha}}.
\end{equation}
Therefore
$$\hat{\varphi}(\xi,t)=\hat{\theta}(\xi)e^{-(t_{0}-t)|\xi|^{\alpha}}=C(\xi)e^{t|\xi|^{\alpha}}.$$
Since
$$\hat{\varphi}_{t}(\xi,t)=|\xi|^{\alpha}\hat{\varphi}(\xi,t),$$
we have that
\begin{equation}
\left\{\begin{matrix}
    \varphi_t-(-\Delta)^{\alpha/2}\varphi=0 &\quad\mbox{for }
    (x,t)\in\mathbb{R}^{n}\times[0,t_0), \\
    \varphi(x,t_0)=\theta(x)&\quad\mbox{in } \mathbb{R}^{n}.
  \end{matrix}\right.\label{pvarphi}
\end{equation}
Now we claim that
\begin{equation}\label{prop1}
|\varphi(x,t)|=|\theta(x)\ast P_{t_0-t}(x)|\leq\frac{C_1}{1+|x|^{n+\alpha}},
\end{equation}
where $C_1$ depends of $n$, $\alpha$, $R_0$, $M:=\|\theta\|_{L^{\infty}(B_{R_0})}$ and $t_0$.
\\Indeed, considering without lost of generality that $2R_0>1$, we will distinguish two cases:

When $|x|\leq 2R_0$, using that $P_t$ is a summability kernel in $L^1$ we have that
\begin{eqnarray}
\left |\int_{\mathbb{R}^{n}}{P_{t_0-t}(y)\theta(x-y)\, dy}\right |&\leq&M\int_{\mathbb{R}^{n}}{P_{t_0-t}(y)\, dy}=M
\nonumber\\ [3mm]
&\leq&   M\,  \frac{1+(2R_0)^{n+\alpha}}{1+|x|^{n+\alpha}}\nonumber\\[3mm]
&\leq&\frac{{c_1}}{1+|x|^{n+\alpha}}.\label{extra0}
\end{eqnarray}
where $c_1=c_1(n,\, \alpha,\, R_0,\, M)$.

Consider now $|x|> 2R_0$. Note that from \eqref{nucleot} and \eqref{P}, we obtain that
\begin{equation}\label{ALM}
P_t(y)\le\frac{C}{t^{\frac{n}{\alpha}}
\left(1+\frac{|y|^{n+\alpha}}{t^{\frac{n+\alpha}{\alpha}}}\right)}
\le \frac{Ct}{|y|^{n+\alpha}}.
\end{equation}
Then, since for $|x-y|\leq R_0$ it follows that $|y|\geq \frac{|x|}{2}$, we have
$$P_{t_0-t}(y)\le\frac{2^{n+\alpha}C(t_0-t)}{|x|^{n+\alpha}}.$$
As a consequence,
\begin{eqnarray}
\left |\int_{\mathbb{R}^{n}}{P_{t_0-t}(y)\theta(x-y)\, dy}\right |&=&\left | \int_{|x-y|\leq R_0}{P_{t_0-t}(y)\theta(x-y)\, dy}\right |\nonumber\\
&\leq&2^{n+\alpha} C M\,|B_{R_0}|\frac{t_0-t}{|x|^{n+\alpha}}\nonumber\\
&\leq&2^{n+\alpha}2C M\,|B_{R_0}|\frac{t_0}{1+|x|^{n+\alpha}}\nonumber\\
&\leq&\frac{c_2}{1+|x|^{n+\alpha}},\label{extra2}
\end{eqnarray}
where
$c_2=c_2(n,\, \alpha,\, R_0,\, M,\, t_0)$
and $|B_{R_0}|$ denotes, as usual, the Lebesgue measure of the ball.
Hence, \eqref{prop1} follows
from \eqref{extra0} and \eqref{extra2}.

\

Also, applying \eqref{extra2} to the derivatives of
$\theta\in \mathcal{C}_{0}^{\infty}(B_{R_0})$, we also have
\begin{equation}\label{prop2}
|\nabla\varphi(x,t)|=|\nabla\theta(x)\ast P_{t_0-t}(x)|\leq\frac{C_2}{1+|x|^{n+\alpha}},
\end{equation}
where $C_2$ depends also of $n$, $\alpha$, $R_0$, $M':=\|\nabla\theta\|_{L^{\infty}(B_{R_0})}$ and $t_0$.

\

Then, from \eqref{prop1} and the fact that $u\in L^{1}([0,t_0],\mathcal{L}^{\alpha/2}(\mathbb{R}^{n}))$ we deduce that
\begin{equation}\label{prop3}
M_2:=\int_{0}^{t_0}{\int_{\mathbb{R}^{n}}{|u(x,t)\varphi(x,t)|\, dx\, dt}}<\infty.
\end{equation}
Let now $\phi\in \mathcal{C}^{\infty}(\mathbb{R})$ be such that \begin{equation}
\label{AA}
\chi_{B_{1/2}}\leq\phi\leq\chi_{B_1}.\end{equation}
For $R>2R_0$ we define $\phi_{R}(x)=\phi\left(\frac{x}{R}\right)$ and
$$\psi(x,t):=\varphi(x,t)\phi_{R}(x),$$
as a test function of problem \eqref{2.00}.
Also, by an explicit computation, one sees that,
for any
$f,\, g\in\mathcal{L}^{\alpha/2}(\mathbb{R}^{n})$,
$$(-\Delta)^{\alpha/2}(fg)(x)=f(x)(-\Delta)^{\alpha/2}g(x)+g(x)(-\Delta)^{\alpha/2}f(x)-B(f,g)(x)$$
with $B(f,g)$ the bilinear form given by
$$B(f,g)(x):=C(n,\alpha) \, \int_{\mathbb{R}^{n}}{\frac{(f(x)-f(y))(g(x)-g(y))}{|x-y|^{n+\alpha}}\, dy}.$$
Appling this formula (for a fixed $t$) to the functions $\varphi$ and $\phi_R$
and recalling \eqref{pvarphi}
we obtain
\begin{equation}\label{a.1} (-\Delta)^{\alpha/2} \psi=\varphi (-\Delta)^{\alpha/2}\phi_R +\phi_R\varphi_t
-B(\varphi,\phi_R).\end{equation}
Moreover, the function $u\psi$ evaluated  at $t_0$ is $u(x,t_0) \theta(x)\phi_R(x)$, thanks to
the terminal time condition in \eqref{pvarphi}.
%We use these observations and \eqref{a.1}
%in considering problem \eqref{2.00} against the test function $\psi$:
%thus we obtain
Thus, considering $\psi$ as a test function in \eqref{2.00} and using that $\varphi$ is a solution of the problem \eqref{pvarphi} we obtain
\begin{eqnarray}
&&  \left | \int_{\mathbb{R}^{n}}{u(x,t_0)\theta(x)\phi_{R}(x)\, dx}\right |\\
&=&
\left | \int_{0}^{t_0}{\int_{\mathbb{R}^{n}}\left [u\varphi_{t}(x,t)\phi_{R}(x)-u(-\Delta)^{\alpha/2}(\psi(x,t))\right]\, dx\, dt}\right |\nonumber\\
&=&
\left |\int_{0}^{t_0}{\int_{\mathbb{R}^{n}}\left[uB(\varphi,\phi_{R})(x,t)-u\varphi(x,t)(-\Delta)^{\alpha/2}\phi_{R}(x)\right]\, dx\, dt} \right |\nonumber\\
&\leq&\int_{0}^{t_0}{\int_{\mathbb{R}^{n}}{|u\varphi(x,t)|\,|(-\Delta)^{\alpha/2}\phi_{R}(x)|\, dx\, dt}}\nonumber \\
&&\qquad+\int_{0}^{t_0}{\int_{\mathbb{R}^{n}}{|u(x,t)|\,|B(\varphi,\phi_{R})(x,t)|\, dx\, dt}}\nonumber
\end{eqnarray}
Then, since ${\theta}$ is supported in $B_{R_0}$ and $R_0<R/2$,
and recalling \eqref{AA},
we  conclude that
\begin{eqnarray}
\left |\int_{\mathbb{R}^{n}}{u(x,t_0)\theta(x)\, dx} \right|  &\leq&\int_{0}^{t_0}{\int_{\mathbb{R}^{n}}{|u\varphi(x,t)|
\,|(-\Delta)^{\alpha/2}\phi_{R}(x)|\, dx\, dt}}\nonumber\\
&+&\int_{0}^{t_0}{\int_{\mathbb{R}^{n}}{|u(x,t)|\,|B(\varphi,\phi_{R})(x,t)|\, dx\, dt}}\nonumber\\
&=:&I_1(R)+C(n,\alpha)\,I_2(R).\label{step2}
\end{eqnarray}
It remains to show that
$$\lim_{R\to \infty}{I_1(R)+I_{2}(R)}=0.$$
Indeed, since
$$|(-\Delta)^{\alpha/2}\phi_{R}(x)|=R^{-\alpha}\left|
\left((-\Delta)^{\alpha/2}\phi\right)\left(\frac{x}{R}\right)\right|\leq C_0R^{-\alpha},$$
by \eqref{prop3}, it follows that
$$
I_1(R)\leq C_0R^{-\alpha}\int_{0}^{t_0}{\int_{\mathbb{R}^{n}}{|u\varphi(x,t)|\, dx\, dt}}\leq C_0M_2R^{-\alpha}
$$
and so
\begin{equation}\label{step3}
\lim_{R\to\infty}{I_1(R)}=0.
\end{equation}
Now we are going to estimate $I_2(R)$.
For this we cover $\mathbb{R}^n\times\mathbb{R}^n$ with six domains suitably described
by the radii $R/4$, $R/2$, $R$ and $2R$ and represented  (for $n=1$) in the following picture:
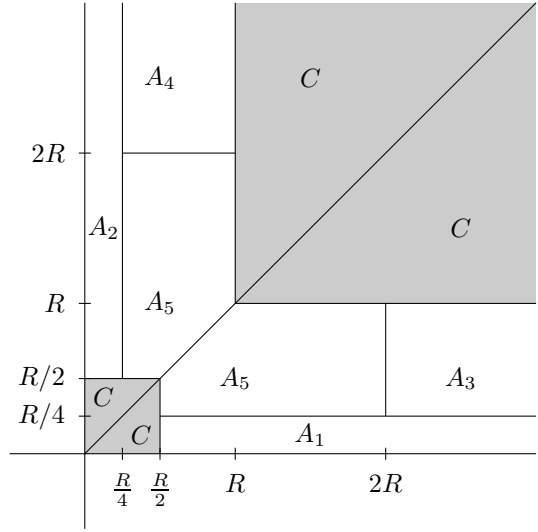
\begin{center}
\begin{tikzpicture}
\def\R{2}
\coordinate (O) at (0,0);

\fill [color=black!20] (O)--(\R/2,0)--(\R/2,\R/2)--(0,\R/2)--cycle;

\fill [color=black!20] (\R,3*\R)--(\R,\R)--(3*\R,\R)--(3*\R,3*\R)--cycle;

\draw (-1,0)--(3*\R,0);
\draw (0,-1)--(0,3*\R);
\draw (O)--(3*\R,3*\R);
\draw (0,\R/2)--(\R/2,\R/2)--(\R/2,0);
\draw (\R/2,\R/4)--(3*\R,\R/4);
\draw (2*\R,\R/4)--(2*\R,\R);
\draw (\R,3*\R)--(\R,\R)--(3*\R,\R);
\draw (\R/4,\R/2)--(\R/4,3*\R);
\draw (\R/4,2*\R)--(\R,2*\R);

\draw ($(\R/4,0)+(0,-\R/25)$)--($(\R/4,0)+(0,+\R/25)$);
\draw (\R/4,0) node [below,yshift=-4pt] {$\frac R4$};

\draw ($(\R/2,0)+(0,-\R/25)$)--($(\R/2,0)+(0,+\R/25)$);
\draw (\R/2,0) node [below,yshift=-4pt] {$\frac R2$};

\draw ($(\R,0)+(0,-\R/25)$)--($(\R,0)+(0,+\R/25)$);
\draw (\R,0) node [below,yshift=-4pt] {$R$};

\draw ($(2*\R,0)+(0,-\R/25)$)--($(2*\R,0)+(0,+\R/25)$);
\draw (2*\R,0) node [below,yshift=-4pt] {$2R$};

\draw ($(0,\R/4)+(-\R/25,0)$)--($(0,\R/4)+(\R/25,0)$);
\draw (0,\R/4) node [left,xshift=-4pt] {$R/4$};

\draw ($(0,\R/2)+(-\R/25,0)$)--($(0,\R/2)+(\R/25,0)$);
\draw (0,\R/2) node [left,xshift=-4pt] {$R/2$};

\draw ($(0,\R)+(-\R/25,0)$)--($(0,\R)+(\R/25,0)$);
\draw (0,\R) node [left,xshift=-4pt] {$R$};

\draw ($(0,2*\R)+(-\R/25,0)$)--($(0,2*\R)+(\R/25,0)$);
\draw (0,2*\R) node [left,xshift=-4pt] {$2R$};

\draw (3*\R/2,\R/8) node {$A_1$} (\R/8,3*\R/2) node {$A_2$} (5*\R/2,\R/2) node {$A_3$} (\R/2,5*\R/2) node {$A_4$} (\R/2,\R) node {$A_5$} (\R,\R/2) node {$A_5$}  (3*\R/2,5*\R/2) node {$C$}   (5*\R/2,3*\R/2) node {$C$}  (3*\R/8,\R/8) node {$C$}   (\R/8,3*\R/8) node {$C$};

\end{tikzpicture}
\end{center}
Therefore
$$\mathbb{R}^{2n}=\left(\bigcup_{k=1}^{5}{A_{k}}\right)\cup C,$$
where
$$A_1:=\{(x,y):\, |x|> R/2,\, |y|\leq R/4\},\quad A_2:=\{(x,y):\, |x|\leq R/4,\, |y|> R/2\},$$
$$A_3:=\{(x,y):\, |x|\geq 2R,\, R/4<|y|< R\},\quad A_4:=\{(x,y):\, R/4<|x|< R,\, |y|\geq 2R \},$$
$$A_5:=\{(x,y):\, R/4<|x|< 2R,\, R/4<|y|< 2R\}$$
and
$$C:=\{(x,y):\, |x|\leq R/2, \,|y|\leq R/2\}\cup \{(x,y):\, |x|\geq R, |y|\geq R\}.$$

{F}rom \eqref{AA}, we know that $\phi_R(x)-\phi_R(y)=0$ if $(x,y)\in C$,
and so
\begin{eqnarray}
I_{2}(R)&=&\int_{0}^{t_0}{\int_{\mathbb{R}^{n}}{|u(x,t)|
\int_{\mathbb{R}^{n}}{\frac{|\varphi(x,t)-\varphi(y,t)|\,|
\phi_{R}(x)-\phi_{R}(y)|}{|x-y|^{n+\alpha}}\, dy}\, dx\, dt}}\nonumber\\
&\le&\sum_{k=1}^{5}{I_{2}^{A_k}(R)},\label{step4}
\end{eqnarray}
where
$$I_{2}^{A_k}(R)=\int_{0}^{t_0}{\int_{A_k}{|u(x,t)|\frac{|
\varphi(x,t)-\varphi(y,t)|\,|\phi_{R}(x)-\phi_{R}(y)|}{|x-y|^{n+\alpha}}\, dy}\, dx\, dt},$$
for $k=1,\ldots,5$.

We are going to estimate each of these five integral separately. For $(x,y)\in A_1$ we get that $|x-y|\geq C|x|$. Moreover by \eqref{prop1} it follows that
$$|\varphi(x,t)|+|\varphi(y,t)|\leq \frac{
C}{1+|y|^{n+\alpha}}.$$
Therefore
\begin{eqnarray}
I_{2}^{A_1}(R)
&\leq&
\int_{0}^{t_0}{\int_{|x|> R/2}{\frac{|u(x,t)|}{|x|^{n+\alpha}}\int_{|y|\leq R/4}{\frac{C}{1+|y|^{n+\alpha}}\, dy}\, dx\, dt}}\nonumber\\
&\leq&C\int_{0}^{t_0}{\int_{|x|> R/2}{\frac{|u(x,t)|}{|x|^{n+\alpha}}\, dx\, dt}}.\label{step5}
\end{eqnarray}
Following the same ideas, since for $(x,y)\in A_2$ we obtain that
\begin{equation}\label{nuevo}
|\varphi(x,t)|+|\varphi(y,t)|\leq \frac{
C}{|x|^{n+\alpha}},\quad \mbox{and}\quad |x-y|\geq C|y|,
\end{equation}
then
\begin{eqnarray}
I_{2}^{A_2}(R)
&\leq&
\int_{0}^{t_0}{\int_{|x|\leq R/4}{\frac{|u(x,t)|}{|x|^{n+\alpha}}\int_{|y|>R/2}{\frac{C}{|y|^{n+\alpha}}\, dy}\, dx\, dt}}\nonumber\\
&\leq&CR^{-\alpha}\int_{0}^{t_0}{\int_{|x|\le R/4}{\frac{|u(x,t)|}{|x|^{n+\alpha}}\, dx\, dt}}.\label{step6}
\end{eqnarray}
Also, since for $(x,y)\in A_3$ \eqref{nuevo} is satisfied, then
\begin{equation}\label{step7}
I_{2}^{A_3}(R)\leq CR^{-\alpha}\int_{0}^{t_0}{\int_{|x|\geq 2R}{\frac{|u(x,t)|}{|x|^{n+\alpha}}\, dx\, dt}}.
\end{equation}
Similarly,
using again the good decay of $\varphi$ and the fact that $|x-y|\ge C|y|$ for every
$(x,y)\in A_4$, we obtain that
\begin{equation}\label{step8}
I_{2}^{A_4}(R)\leq CR^{-\alpha}\int_{0}^{t_0}{\int_{R/4<|x|< R}{\frac{|u(x,t)|}{|x|^{n+\alpha}}\, dx\, dt}}.
\end{equation}
Then, using the Monotone Convergence Theorem, by \eqref{prop3} and the fact that $u\in L^{1}([0,t_0],\mathcal{L}^{\alpha/2}(\mathbb{R}^{n}))$, from \eqref{step5}, \eqref{step6}, \eqref{step7} and \eqref{step8} it follows that
\begin{equation}\label{IA1234}
\lim_{R\to\infty}I_{2}^{A_1}(R)=\lim_{R\to\infty}I_{2}^{A_2}(R)=\lim_{R\to\infty}
I_{2}^{A_3}(R)=\lim_{R\to\infty}I_{2}^{A_4}(R)=0.
\end{equation}
To estimate $I_{2}^{A_5}(R)$ we will treat separately the cases $\alpha\in[0,1)$ and $\alpha\in[1,2)$.

\

We start with the case $\alpha\in[0,1)$ and $(x,y)\in A_5$. Since in $A_5$
the roles of $x$ and $y$ are symmetric, we deduce from \eqref{prop1} that,
in this case,
\begin{equation}\label{b}
|\varphi(x,t)|+|\varphi(y,t)|\leq \frac{C}{|x|^{n+\alpha}}.
\end{equation}
Also
$$|\phi_{R}(x)-\phi_{R}(y)|\leq \frac{C}{R}|x-y|.$$
Thus
\begin{equation}\label{step9}
I_{2}^{A_5}(R)\leq \frac{C}{R}\int_{0}^{t_0}{\int_{\frac{R}{4}\leq|x|\leq 2R}{\frac{|u(x,t)|}{|x|^{n+\alpha}}\int_{\frac{R}{4}\leq|y|\leq 2R}{\frac{1}{|x-y|^{n+\alpha-1}}\, dy}\, dx\, dt}}.
\end{equation}
By the change of variables $\tilde y:=x-y$, it follows from \eqref{step9} that
\begin{eqnarray}
I_{2}^{A_5}(R)
&\leq&
\frac{C}{R}\int_{0}^{t_0}{\int_{\frac{R}{4}\leq|x|\leq 2R}{\frac{|u(x,t)|}{|x|^{n+\alpha}}\int_{\frac{R}{4}\leq|x-\tilde y|\leq 2R}{\frac{1}{|\tilde y|^{n+\alpha-1}}\, d\tilde y}\, dx\, dt}}\nonumber\\
&\leq&\frac{C}{R}\int_{0}^{t_0}{\int_{\frac{R}{4}\leq|x|\leq 2R}{\frac{|u(x,t)|}{|x|^{n+\alpha}}\int_{|\tilde y|\leq 4R}{\frac{1}{|\tilde y|^{n+\alpha-1}}\, d\tilde y}\, dx\, dt}}.\nonumber\\
&\leq&
CR^{-\alpha}\int_{0}^{t_0}{\int_{\frac{R}{4}\leq|x|\leq 2R}{\frac{|u(x,t)|}{|x|^{n+\alpha}}\, dx\, dt}}.\label{step10}
\end{eqnarray}
Therefore, using that $u\in L^{1}([0,t_0],\mathcal{L}^{\alpha/2}(\mathbb{R}^{n}))$,  we conclude that
\begin{equation}\label{step11}
\lim_{R\to\infty}I_{2}^{A_5}(R)=0,\, \qquad{\mbox{ when }}\alpha\in[0,1).
\end{equation}
We consider now the case $\alpha\in[1,2)$. By \eqref{prop2}, we get that
\begin{equation}\label{step111}
|\varphi(x,t)-\varphi(y,t)|\leq\frac{C}{1+|z|^{n+\alpha}}|x-y|,
\end{equation}
for some $z$  in the segment joining $x$ and $y$.
We take the family
$$Q:=\left\{(x,y)\in A_5:\, |x-y|\leq \frac{R}{100}\right\}.$$
Note that, if $(x,y)\in Q$ then every point $z$ lying on the segment from $x$
to $y$ verifies $|z|\geq C|x|$. Hence,  \eqref{step111} and the previous estimate for $\phi_R$, gives that
\begin{eqnarray}
&&\int_{0}^{t_0}{\int_{(x,y)\in Q}{|u(x,t)|\frac{|(\phi_{R}(x)-\phi_{R}(y))(\varphi(x,t)-\varphi(y,t))|}{|x-y|^{n+\alpha}}\, dy\, dx\, dt}}\nonumber\\
&\leq&\int_{0}^{t_0}{\int_{(x,y)\in Q}{|u(x,t)|\frac{C}{R|x|^{n+\alpha}|x-y|^{n+\alpha-2}}\, dy\, dx\, dt}}\nonumber\\
&\leq&\frac{C}{R}\int_{0}^{t_0}{\int_{\frac{R}{4}\leq|x|\leq 2R}{\frac{|u(x,t)|}{|x|^{n+\alpha}}\int_{\frac{R}{4}\leq|y|\leq 2R}{\frac{1}{|x-y|^{n+\alpha-2}}\, dy}\, dx\, dt}}\nonumber\\
&\leq&\frac{C}{R}\int_{0}^{t_0}{\int_{\frac{R}{4}\leq|x|\leq 2R}{\frac{|u(x,t)|}{|x|^{n+\alpha}}\int_{\frac{R}{4}\leq|x-\tilde y|\leq 2R}{\frac{1}{|\tilde y|^{n+\alpha-2}}\, d\tilde y}\, dx\, dt}}\nonumber\\
&\leq&CR^{1-\alpha}\int_{0}^{t_0}{\int_{\frac{R}{4}\leq|x|\leq 2R}{\frac{|u(x,t)|}{|x|^{n+\alpha}}\, dx\, dt}}.\label{step12}
\end{eqnarray}
On the other hand, if %$(x,y)\in A_5\setminus \bigcup_{k\in I}Q_{k}$ we have that
$(x,y)\in A_5\setminus Q$ we have that
\begin{equation}\label{c}
|x-y|>\frac{R}{100}\geq C|y|.
\end{equation}
Then by \eqref{b} and \eqref{c} it follows that
\begin{eqnarray}
&&\int_{0}^{t_0}{\int_{(x,y)\in A_5\setminus Q}{|u(x,t)|\frac{|(\phi_{R}(x)-\phi_{R}(y))(\varphi(x,t)-\varphi(y,t))|}{|x-y|^{n+\alpha}}\, dy\, dx\, dt}}\nonumber\\
&\leq&\frac{C}{R}\int_{0}^{t_0}{\int_{(x,y)\in A_5\setminus Q}{|u(x,t)|\frac{|(\varphi(x,t)-\varphi(y,t))|}{|x-y|^{n+\alpha-1}}\, dy\, dx\, dt}}\nonumber\\
&\leq&\frac{C}{R}\int_{0}^{t_0}{\int_{(x,y)\in A_5\setminus Q}{\frac{|u(x,t)|}{|x|^{n+\alpha}}\frac{1}{|y|^{n+\alpha-1}}\, dy\, dx\, dt}}\nonumber\\
&\leq&\frac{C}{R}\int_{0}^{t_0}{\int_{\frac{R}{4}\leq|x|\leq 2R}{\frac{|u(x,t)|}{|x|^{n+\alpha}}\int_{\frac{R}{4}\leq|y|\leq 2R}{\frac{1}{|y|^{n+\alpha-1}}\, dy}\, dx\, dt}}\nonumber\\
&\leq&CR^{-\alpha}\int_{0}^{t_0}{\int_{\frac{R}{4}\leq|x|\leq 2R}{\frac{|u(x,t)|}{|x|^{n+\alpha}}\, dx\, dt}}.\label{step13}
\end{eqnarray}
Therefore, from \eqref{step12} and \eqref{step13}
\begin{eqnarray}
I_{2}^{A_5}(R)&\leq&\int_{0}^{t_0}{\int_{(x,y)\in Q}{|u(x,t)|\frac{|(\phi_{R}(x)-\phi_{R}(y))(\varphi(x,t)-\varphi(y,t))|}{|x-y|^{n+\alpha}}\, dy\, dx\, dt}}\nonumber\\
&+&\int_{0}^{t_0}{\int_{(x,y)\in A_5\setminus Q}{|u(x,t)|\frac{|(\phi_{R}(x)-\phi_{R}(y))(\varphi(x,t)-\varphi(y,t))|}{|x-y|^{n+\alpha}}\, dy\, dx\, dt}}\nonumber\\
&\leq&CR^{1-\alpha}\int_{0}^{t_0}{\int_{\frac{R}{4}\leq|x|\leq 2R}{\frac{|u(x,t)|}{|x|^{n+\alpha}}\, dx\, dt}}\nonumber\\
&+&CR^{-\alpha}\int_{0}^{t_0}{\int_{\frac{R}{4}\leq|x|\leq 2R}{\frac{|u(x,t)|}{|x|^{n+\alpha}}\, dx\, dt}}.
\end{eqnarray}
Then, since $u\in L^{1}([0,t_0],\mathcal{L}^{\alpha/2}(\mathbb{R}^{n}))$, using the Monotone Convergence Theorem we obtain
\begin{equation}\label{step14}
\lim_{R\to\infty} I_{2}^{A_5}(R)=0,\, \qquad{\mbox{ when }}\alpha\in[1,2).
\end{equation}
That is, by \eqref{step11} and \eqref{step14}, we get
\begin{equation}\label{IA5}
\lim_{R\to\infty}|I_{2}^{A_5}(R)|=0,\, \qquad{\mbox{ whenever }}\alpha\in(0,2).
\end{equation}
Putting together \eqref{step4},  \eqref{IA1234} and \eqref{IA5} it follows that
\begin{equation}\label{step15}
\lim_{R\to\infty}I_{2}(R)=0,\, \qquad{\mbox{ when }}\alpha\in(0,2).
\end{equation}
Therefore, from \eqref{step2}, by \eqref{step3} and \eqref{step15} we conclude that
$$\lim_{R\to\infty}\int_{\mathbb{R}^{n}}{u(x,t_0)\theta(x)\, dx}=0,\, $$
for an arbitrary $\theta\in \mathcal{C}_{0}^{\infty}(B_{R_0}), \, R_0<2R.$
\end{proof}
%%%%%%%%%%%%%%%%%%%%%%%%%  THE END OF PROOF OF TH 1
%%%%%%%%%%%%%%%%%%%%%%%%%
\section{Uniqueness for strong positive solutions}\label{postrong}
%To finish we are going to present an important consequence of Theorem \ref{Th1} and that improve Lemma \ref{menor}.
In this section we will establish the representation of the positive
strong solutions of the fractional heat equation as the Poisson integral of the initial value. That is
\begin{theorem}\label{Th2}
Let $(x,t)\in\mathbb{R}^{n }\times(0,T)$. If $u(x,t)\geq0$ is a strong solution of the fractional heat equation, then
$$\int_{\mathbb{R}^{n}}{P_{t}(x-y)u(y,0)\, dy}=u(x,t).$$
\end{theorem}

To prove this theorem we will need some previous results that we present as follows.
First of all, we establish that, among all possible positive solutions of the fractional heat equation, the smallest one is given by a formula that involves the convolution with the fractional heat kernel (see Lemma \ref{menor}). To prove it we will use  the following
%%%%%%%%%%%%%  WEAK MAXIMUM PRINCIPLE LEMMA
\begin{lemma}[A maximum principle]\label{weak}
Set $D_{T}:=\Omega\times (0,T)$ and let $v(x,t)\in \mathcal{C}(\overline{\Omega}\times [0,T))$  satisfy, pointwise, the following problem
\begin{equation}\label{o0}
\left\{\begin{array}{ll}
    v_t+(-\Delta)^{\alpha/2}v\leq0 &\quad\mbox{for }
    (x,t)\in D_{T}, \\
    v(x,t)\leq 0&\quad\mbox{in } \big(\mathbb{R}^{n}\times[0,T)\big)\setminus D_{T}.
  \end{array}\right.
\end{equation}
Then $v\le0$ in $\overline{\Omega}\times [0,T)$.
\end{lemma}
%%%
\begin{proof}
Fixing an arbitrary $T'\in(0,T)$, we define
$$v(x_0,t_0):=\max_{\overline{\Omega}\times [0,T']} v(x,t).$$
Our goal is to show that $v(x_0,t_0)\le0$.
The proof is by contradiction, assuming that
\begin{equation}\label{88}
v(x_0,t_0)>0.\end{equation}
Then, $(x_0,t_0)$ cannot lie in $(\partial\Omega\times[0,T))\cup (\Omega\times\{0\})$, since $v\leq 0$ there, thanks to the boundary conditions
in \eqref{o0}. As a consequence,
$(x_0,t_0)$ lies in $\Omega\times(0,T']$ and therefore $v_{t}(x_0,t_0)=0$. Therefore the equation in \eqref{o0} implies that
\begin{eqnarray*}
0&\geq&(-\Delta)^{\alpha/2}v(x_0,t_0)=C_{n,\alpha} \mbox{P.V.}\int_{\mathbb{R}^{n}}{\frac{v(x_{0},t_{0})-v(y,t_{0})}{|y-x_0|^{n+\alpha}}\, dy}\\
&=&C_{n,\alpha}\Big(\mbox{P.V.}\int_{\Omega}{\frac{v(x_{0},t_{0})-v(y,t_{0})}{|y-x_0|^{n+\alpha}}\, dy}
+\int_{\mathbb{R}^{n}\setminus\Omega}{\frac{v(x_{0},t_{0})-v(y,t_0)}{|y-x_0|^{n+\alpha}}\, dy}\Big)\\
&\geq&C_{n,\alpha} \mbox{P.V.}\int_{\mathbb{R}^{n}\setminus\Omega}{\frac{v(x_{0},t_{0})-v(y,t_0)}{|y-x_0|^{n+\alpha}}\, dy}.
\end{eqnarray*}
Since $v(y,t_0)\leq0$ for $y\in\mathbb{R}^{n}\setminus\Omega$,
thanks to \eqref{o0}, we obtain that the latter integrand is strictly positive,
due to \eqref{88}, and this
is a contradiction.
\end{proof}

%%%%%%%%%%%%%  END OF WEAK MAXIMUM PRINCIPLE LEMMA

Now we are able to prove that strong solutions are upper bounds
for the kernel convolutions:

%%%%%%%%%%%%%%%%%%%%%%%%

\begin{lemma}\label{menor}
Let $(x,t)\in\mathbb{R}^{n }\times(0,T)$. If $u(x,t)\geq0$ is a strong solution of the fractional heat equation then
\begin{equation}\label{cFs}
I:=\int_{\mathbb{R}^{n}}{P_{t}(x-y)u(y,0)\, dy}\leq u(x,t),
\end{equation}
where $P_{t}(x)$ is the function defined in \eqref{nucleot}.
\end{lemma}
\begin{proof}
First of all, observe that the integral $I=I(x,t)$ exists, for is given by the integration of two (measurable) positive functions, although  we do not know a priori  that $I$ is finite. However, this will be a consequence of our result that gives the inequality  $I\leq u(x,t)$ for every $(x,t)\in\mathbb{R}^{n}\times[0,T)$. To this aim, we let
\begin{equation}
\phi_{R}(x):=\left\{\begin{array}{ll}
    1, &\quad |x|\leq R-1, \\
    R-|x|,&\quad R-1\leq|x|\leq R,\\
    0,&\quad |x|>R.
  \end{array}\right.
\end{equation}
We define
$$v_{R}(x,t):=\int_{\mathbb{R}^{n}}{P_{t}(x-y)\phi_{R}(y)u(y,0)\, dy}=(P_{t}(\cdot)\ast \phi_{R}u(\cdot,0))(x).$$
Then
$$
\left\{\begin{array}{ll}
    \frac{\partial v_{R}}{\partial t}+(-\Delta)^{\alpha/2}v_{R}=0 &\quad\mbox{for }
    (x,t)\in\mathbb{R}^{n}\times(0,T), \\
    v_{R}(x,t)\geq 0 &\quad\mbox{for }
    (x,t)\in\mathbb{R}^{n}\times(0,T), \\
    v_{R}(x,0)=\phi_{R}(x)u(x,0)&\quad\mbox{in } \mathbb{R}^{n}.
  \end{array}\right.
$$
Let $|x|>R$. As $u(x,t)\in \mathcal{C}(\mathbb{R}^{n}\times[0,T))$ we can define the real number \begin{equation*}
M_{R}:=\sup_{|y|<R}{u(y,0)}<\infty.\end{equation*}
By \eqref{ALM} we have that
\begin{eqnarray*}
0\leq v_{R}(x,t)&\leq& M_{R}\int_{B_R}{P_{t}(x-y)\, dy}\\
&\leq& CM_{R}\int_{B_{R}}{\frac{T}{|x-y|^{n+\alpha}}\, dy}\\
&\leq& C(T,M_{R})\int_{B_{R}}{\frac{\, dy}{||x|-R|^{n+\alpha}}}\\
&=& C(T,M_{R},n)\frac{R^{n}}{||x|-R|^{n+\alpha}},\end{eqnarray*}
for any~$(x,t)\in(\mathbb{R}^{n}\setminus B_{R})\times(0,T)$.

Then, for every $\varepsilon>0$ it follows that
\begin{equation}\label{gardel}
0\leq v_{R}(x,t)\leq \varepsilon,\, \quad
{\mbox{ for any }}
|x|\geq\rho,\, t\in(0,T),
\end{equation}
where
$$\rho=R+\left(\frac{C(T,M_{R},n)R^n}{\varepsilon}\right)^{\frac{1}{n+\alpha}}>0.$$
\\
Moreover, as $u(x,t)\geq 0$ in $\mathbb{R}^{n}\times[0,T)$ we obtain that
\begin{equation}\label{condicion1}
0\leq v_{R}(x,t)\leq\varepsilon\leq\varepsilon+u(x,t),\,
{\mbox{ for any }}|x|\geq\rho, t\in(0,T)
\end{equation}
and
\begin{equation}\label{condicion2}
v_{R}(x,0)=\phi_{R}(x)u(x,0)\leq u(x,0)\leq \varepsilon+u(x,0),\,
{\mbox{ for any }}|x|\leq\rho.
\end{equation}

\

\noindent Consider the cylinder
$$D_{\rho,T}=B_{\rho}\times(0,T).$$
We define the function
$$w(x,t):=v_{R}(x,t)-u(x,t)-\varepsilon.$$
Then, by \eqref{condicion1} and \eqref{condicion2}, we get that $w(x,t)\leq 0$ in $\mathbb{R}^{n}\times[0,T)\setminus D_{\rho,T}$. Therefore, since $u(x,t)$ is a strong solution of the fractional heat equation in $\mathbb{R}^{n}\times[0,T)$, applying  Lemma \ref{weak} in $D_{\rho,T}$ to the function $w(x,t)$ we have that
$$v_{R}(x,t)\leq \varepsilon+u(x,t),\quad\mbox{for $|x|\leq\rho$ and $t\in[0,T)$}.$$
Therefore, from \eqref{condicion1}, it follows that
$$v_{R}(x,t)\leq \varepsilon+u(x,t),\quad\mbox{for $x\in\mathbb{R}^{n}$ and $t\in[0,T)$.}$$
Since $\varepsilon$ is fixed but arbitrary, the previous inequality implies that
$$v_{R}(x,t)\leq u(x,t),\, \mbox{for every $x\in\mathbb{R}^{n}$ and $t\in[0,T).$}$$
\\Finally, by the Monotone Convergence Theorem, as $\displaystyle{\lim_{R\to\infty}{\phi_{R}}=1}$, we conclude that
\begin{equation*}
0\leq v(x,t)=\lim_{R\to\infty}{v_R(x,t)}=\int_{\mathbb{R}^{n}}{P_{t}(x-y)u(y,0)dy}\leq u(x,t).\qedhere\end{equation*}
\end{proof}

\

\

By a simple time translation, we obtain from Lemma \ref{menor} the following:
%%%%%%%%%%%%%%%%%%%%%%%
\begin{corollary}\label{corol}
Let $0<\tau<T$ and $(x,t)\in\mathbb{R}^{n }\times(0,T-\tau)$. If $u(x,t)\geq0$ is a strong solution of the fractional heat equation then
\begin{equation}\label{corol2}
\int_{\mathbb{R}^{n}}{P_{t}(x-y)u(y,\tau)\, dy}\leq u(x,t+\tau),
\end{equation}
where $P_{t}(x)$ is the function defined in \eqref{nucleot}. As a consequence, for every $x\in\mathbb{R}^{n}$ and $t\in [0, T-\tau)$ we have
\begin{equation}\label{imp1}
\int_{0}^{T-t}{\int_{\mathbb{R}^{n}}{P_{t}(x-y)u(y,\tau)\, dy\, d\tau}}\leq \int_{0}^{T-t}{u(x,t+\tau)d\tau}.
\end{equation}
\end{corollary}

\

\

Moreover we have the following
\begin{corollary}\label{corol3}
Let $(x,t)\in\mathbb{R}^{n }\times(0,T)$. If $u(x,t)\geq0$ is a strong solution of the fractional heat equation then, $u(\cdot, t)\in\mathcal{L}^{\alpha/2}(\mathbb{R}^{n})$.
\end{corollary}
\begin{proof}
Let $0<T'<T$. Taking $t=T-T'$ in \eqref{corol2}, from \eqref{P}, we get that
\begin{eqnarray}\label{imp2}
\frac{T-T'}{C}\int_{\mathbb{R}^{n}}{\frac{u(y,\tau)}{(T-T')^{\frac{n+\alpha}{\alpha}}+|x-y|^{n+\alpha}}\, dy}&\leq& \int_{\mathbb{R}^{n}}{P_{T-T'}(x-y)u(y,\tau)\, dy}\nonumber\\
&\leq& u(x,T-T'+\tau)<\infty.
\end{eqnarray}
Let
$$C(T):=\min\left\{1, \frac{1}{(T-T')^{\frac{n+\alpha}{\alpha}}}\right\}.$$
Then, since
$$\frac{1}{(T-T')^{\frac{n+\alpha}{\alpha}}+|y|^{n+\alpha}}\geq C(T)\frac{1}{1+|y|^{n+\alpha}},$$
taking $x=0$ in \eqref{imp2}, we conclude that
\begin{equation*}
\int_{\mathbb{R}^{n}}{\frac{|u(y,t)|}{(1+|y|^{n+\alpha})}\, dy}<\infty.
\qedhere
\end{equation*}
\end{proof}

\

\

We also have
\begin{corollary}\label{corol4}
Let $(x,t)\in\mathbb{R}^{n }\times(0,T)$. If $u(x,t)\geq0$ is a strong solution of the fractional heat equation, then $u(\cdot, t)\in L^{1}([0,T'],\mathcal{L}^{\alpha/2}(\mathbb{R}^{n}))$ for every $0<T'<T.$
\end{corollary}
\begin{proof}
Take an arbitrary $0<T'<T$. By \eqref{imp1} with $t=T-T'$ we have that
$$\int_{0}^{T'}{\int_{\mathbb{R}^{n}}{P_{T-T'}(x-y)u(y,\tau)\, dy\, d\tau}}\leq \int_{0}^{T'}{u(x,\tau)d\tau}.$$
Then, as $u\in \mathcal{C}(\mathbb{R}^{n}\times[0,T'])$ we get
$$\int_{0}^{T'}{\int_{\mathbb{R}^{n}}{P_{T-T'}(x-y)u(y,\tau)\, dy\, d\tau}}<\infty.$$
Therefore doing as the proof of Corollary \ref{corol3} we get that
$$\int_{0}^{T'}{\int_{\mathbb{R}^{n}}{\frac{|u(y,t)|}{(1+|y|^{n+\alpha})}\, dy\, dt}}<\infty.$$
That is, $u\in L^{1}([0,T'],\mathcal{L}^{\alpha/2}(\mathbb{R}^{n}))$ for every $0<T'<T.$
\end{proof}
Note that Corollary~\ref{corol4}
affirms that if $u(x,t)\geq0$ is a strong solution of the fractional heat equation then $u$  is also a weak solution of the same equation.

\

%%%%%%%%%%%%%%%%%%%%%%%
%%%%%%%%%%%%%%%%%%%%%%% LAST THEOREM
%%%%%%%%%%%%%%%%%%%%%%%
Now we are able to prove our main result:
%%%%%%%%%%%%%%%%%%%%%%%%%

\begin{proof}[Proof of Theorem \ref{MAIN}]
By Corollary \ref{corol4} we get that $u\in L^{1}([0,T'],\mathcal{L}^{\alpha/2}(\mathbb{R}^{n}))$ for every $0<T'<T.$
Moreover if we define
$$p(x,t):=\int_{\mathbb{R}^{n}}{P_{t}(x-y)u(y,0)\, dy},$$
 by \eqref{imp1} with $t=T-T'$, we also have that $p\in L^{1}([0,T'],\mathcal{L}^{\alpha/2}(\mathbb{R}^{n}))$. Let now the function
$$w(x,t):=u(x,t)-p(x,t)\geq0.$$
It is clear that $w$ is a strong solution of the fractional heat equation. Moreover, as $w\in L^{1}([0,T'],\mathcal{L}^{\alpha/2}(\mathbb{R}^{n}))$, then $w(x,t)$ is also a solution in the weak sense with zero initial datum. Therefore applying Theorem \ref{Th1} we conclude that $w(x,t)=0$ for every $(x,t)\in\mathbb{R}^{n}\times[0,T)$.
\end{proof}
%%%%%%%%%%%%%
%%%%%%%%%%%%%
\subsection{About viscosity solutions.}
As was said at the beginning of this work, it is natural to consider viscosity solutions of the fractional heat equation.
Our purpose here is to describe  some cases in which a positive viscosity solution has the unique representation in terms of the kernel $P_t$.
\begin{proposition}\label{viscosity1}
Let $\{u_n\}$ be a sequence of non-negative, strong solutions of the fractional heat equation converging, uniformly over compact sets, to a given function $u$.
Then $u\geq0$ is a viscosity solution of \eqref{eq:base} satisfying
\begin{equation}\label{v1}
\int_{\mathbb{R}^{n}}{P_{t}(x-y)u(y,0)\, dy}\leq u(x,t),\quad (x,t)\in\mathbb{R}^{n}\times(0,T)
\end{equation}
and
\begin{equation}\label{v2}
u\in L^{1}([0,T'],\mathcal{L}^{\alpha/2}(\mathbb{R}^{n}))\quad\mbox{ for every $0<T'<T.$}
\end{equation}
That is, the conclusions of Lemma \ref{menor}  and Corollary \ref{corol4} are satisfied.
\end{proposition}
\begin{proof}
By Lemma \ref{menor} it follows that
$$\int_{\mathbb{R}^{n}}{P_{t}(x-y)u_n(y,0)\, dy}\leq u_n(x,t),\quad (x,t)\in\mathbb{R}^{n}\times(0,T).$$
Applying the Fatou Lemma we obtain \eqref{v1}. Therefore, doing as in the proof of the Corollary \ref{corol4} we conclude \eqref{v2}.
Note also that, by the comparison principle (Corolary 2.1.6 of \cite{S}), $u_n\geq 0$ is a viscosity solution of \eqref{eq:base} for every $n\in\mathbb{N}$. Therefore, since $u$ is the uniform limit over compact sets of viscosity solutions, we get that $u$ is also a viscosity solution of \eqref{eq:base}.
\end{proof}

Notice that to conclude the equality in \eqref{v1} we would need to know that $u$ is a weak solution.

If we add a monotonicity condition over the sequence $\{u_n\}$ then we obtain the following
\begin{proposition}\label{viscosity2}

Let $\{u_n\}$ be a monotone sequence of non-negative,  strong solutions of the fractional heat equation converging, uniformly over compact sets, to a given function $u$. Then $u\geq0$ is a strong solution of \eqref{eq:base} satisfying
\begin{equation}\label{v21}
\int_{\mathbb{R}^{n}}{P_{t}(x-y)u(y,0)\, dy}=u(x,t),\quad (x,t)\in\mathbb{R}^{n}\times(0,T).
\end{equation}
\end{proposition}
\begin{proof}
Since, for every $n\in\mathbb{N}$, $u_n\geq0$ satisfies Theorem \ref{MAIN}, by the Monotone Convergence Theorem we obtain \eqref{v21}.
Clearly this implies that  $u\geq0$ is a strong solution of \eqref{eq:base}.
\end{proof}
\begin{remark}
It would be interesting  to find the biggest class of positive viscosity solution for which the representation property
\eqref{v21} holds. This seems to be an open problem as far as we know.
\end{remark}
%%%%%%%%%%%%%
%%%%%%%%%%%%%
\section{Further results.}
As in the local case (see \cite{W}), given a solution $u$, one can define
the {\it enthalpy}
$$\displaystyle v(x,t)=\int_{0}^{t}{u(x,s)\, ds}$$
which is also a solution of the fractional heat equation. Indeed we have the following result:
%%%%%%%%%%%%%%%%%%%%%%%%%
\begin{lemma}\label{entalpy}
Let $(x,t)\in\mathbb{R}^{n }\times[0,T)$. If $u(x,t)$ is a strong solution of the fractional heat equation with vanishing initial condition, then the entalpy term
$$v(x,t)=\int_{0}^{t}{u(x,s)\, ds}$$
is also a strong solution of the fractional heat equation. Moreover if $u$ is positive the function $v$ is increasing in $t$ for $x$ fixed and $\alpha/2$-subharmonic as a function of $x$ in the weak sense.
\end{lemma}
\begin{proof}
Let $(x,t)\in\mathbb{R}^{n}\times[0,T)$. First of all note that $v(x,t)$ satisfies the conditions i)-iii) of the Definition \ref{strong}. Therefore, since $u(x,t)$ is a strong solution of the fractional heat equation, by the Fundamental Theorem of Calculus and Fubini Theorem it follows that
\begin{eqnarray*}
&& C(n,\alpha){\mbox{P.V.}}\int_{\mathbb{R}^{n}}{\frac{v(x,t)-v(y,t)}{|x-y|^{n+\alpha}}\, dy}
\\&=&C(n,\alpha){\mbox{P.V.}}
\int_{\mathbb{R}^{n}}{\frac{\int_{0}^{t}{u(x,s)\, ds}-\int_{0}^{t}{u(y,s)\, ds}}{|x-y|^{n+\alpha}}\, dy}\\
&=&C(n,\alpha){\mbox{P.V.}}
\int_{\mathbb{R}^{n}}{\int_{0}^{t}{\frac{u(x,s)-u(y,s)}{|x-y|^{n+\alpha}}\, ds\, dy}}\\
&=&\int_{0}^{t}{(-\Delta)^{\alpha/2}u(x,s)\, ds}\\
&=&-\int_{0}^{t}u_{s}(s,t)\, ds\\
&=&-u(x,t)\\
&=&-v_{t}(x,t),\,
\end{eqnarray*}
for any $(x,t)\in\mathbb{R}^{n}\times(0,T]$.
Then, $v(x,t)$ satisfies the fractional heat equation in the strong sense.
Also if $u\geq 0$ from the calculations done before we deduce that
$$-(-\Delta)^{\alpha/2}v(x,t)=v_{t}(x,t)=u(x,t)\geq 0.$$
So $v$ is  $\alpha/2$-subharmonic as a function of $x$ and increasing in $t$ for $x$ fixed.
\end{proof}
%%%%%%%%%%%%%%%%%%%%%%%%%%

Lemma \ref{entalpy} shows that the {enthalpy} belongs to the special
class of positive strong solutions that are also
$\alpha/2$-subharmonic. This class naturally satisfies
a polynomial estimate, as shown by the following result.
%%%%%%%%%%%%%%%%%%%%%%%%%%
\begin{proposition}\label{point}
Let $(x,t)\in\mathbb{R}^{n }\times[0,T)$. Let $u(x,t)\geq0$ such that
\begin{itemize}
\item[i)]{$u(x,t)$ is an $\alpha/2$-subharmonic function with respect to the variable $x$.}
\item[ii)]{$u(x,t)$ is a strong solution of the fractional heat equation.}
\end{itemize}
Then $u(x,t)\leq C(t)(1+|x|^{n+\alpha})$ if $(x,t)\in\mathbb{R}^{n}\times[0,T)$.
\end{proposition}
%%%%%%%%%%%%%%%%%%%%%%%%%
%%%%%%%%%%%%%%%%%%%%%%%%%
\begin{proof}
By hypothesis it is clear that $u$ is an $\alpha/2$-subharmonic function with respect to the variable $x$ for $t$ fixed and therefore
$u$ is increasing in time.
\\Let now $0<t_1<T$ and $0<t_0<T-t_1$. By Corollary \ref{corol} we have that
$$\int_{\mathbb{R}^{n}}{P_{t}(x-y)u(y,t_1)dy}\leq u(x,t+t_1),\,
{\mbox{ for any }}0<t<T-t_1.$$
Therefore
$$M_{t_0}:=\int_{\mathbb{R}^{n}}{P_{t_0}(y)u(y,t_1)dy}\leq u(0,t_0+t_1)<\infty.$$
Our objective is to show that
\begin{equation}\label{objetivo}
|u(x,t_1)|\leq C(1+|x|^{n+\alpha}).
\end{equation}
Once this is done, using that $u$ is increasing in time, we would get
$$0\leq u(x,t)\leq u(x,t_1)\leq C(1+|x|^{n+\alpha})\quad\mbox{for every $(x,t)\in\mathbb{R}^{n}\times(0,t_1)$}.$$ But, since  $t_0$ and $t_1$ are fixed but arbitrary, we would conclude that
$$|u(x,t)|\leq C(1+|x|^{n+\alpha}).$$
So, we are left to showing that \eqref{objetivo} is true. Note that from Corollary \ref{corol3} we have that $u\in\mathcal{L}^{\alpha/2}(\mathbb{R}^{n})$. Then as $u$ is $\alpha/2$-subharmonic, by Proposition 2.2.6 of \cite{S}
$u$ also satisfies the following mean value property:
\begin{equation}\label{meanvalue}
u(x_0,t)\leq \int_{\mathbb{R}^{n}}{\gamma_{\lambda}(y-x_0)u(y,t)dy}\, \mbox{
for every $x_0\in\Omega\subseteq\mathbb{R}^{n}$ and $\lambda\leq\rm{dist}(x_0,\partial\Omega)$},
\end{equation}
where
$$\gamma_{\lambda}(x):=(-\Delta)^{\alpha/2}\Gamma_{\lambda}(x),$$
and
$$\Gamma_{\lambda}(x):=\frac{1}{\lambda^{n-\alpha}}\Gamma\Big(\frac{x}{\lambda}\Big).$$
Here $\Gamma$ is a $C^{1,1}$ function that coincides with $\Phi(x):=c|x|^{\alpha-n}$ outside the ball $B_{\lambda}$ and $\Gamma$ is a paraboloid inside this ball. Note that $\Phi$ is the fundamental solution of $(-\Delta)^{\alpha/2}$.
From \eqref{meanvalue} we have
\begin{eqnarray}
u(x,t_1)(1+|x|^{n+\alpha})^{-1}&\leq&(1+|x|^{n+\alpha})^{-1}\int_{\mathbb{R}^{n}}{\gamma_{\lambda}(y)u(x-y,t_1)\, dy}\nonumber\\
&=&(1+|x|^{n+\alpha})^{-1}\int_{\{y:\, |y|\geq\lambda\}}{\gamma_{\lambda}(y)u(x-y,t_1)\, dy}\nonumber\\
&+&(1+|x|^{n+\alpha})^{-1}\int_{\{y:\, |y|\leq\lambda\}}{\gamma_{\lambda}(y)u(x-y,t_1)\, dy}\nonumber\\
&:=&I_1(x)+ I_2(x).\label{ism}
\end{eqnarray}
Choosing \begin{equation}\label{4.3bis}
\lambda=|x|/4,\end{equation} we have that $|x-y|\leq5|y|$. Therefore, by Proposition 2.2.3 of \cite{S}, we obtain
\begin{eqnarray}
I_1(x)&\leq&C(1+|x|^{n+\alpha})^{-1}\int_{\{y:\, |y|\geq\lambda\}}{\frac{u(x-y,t_1)}{|y|^{n+\alpha}}\, dy}\nonumber\\
&\leq&C(1+|x|^{n+\alpha})^{-1}\int_{\mathbb{R}^{n}}{\frac{u(x-y,t_1)}{|x-y|^{n+\alpha}}\, dy}\nonumber\\
&\leq&C\|u\|_{\mathcal{L}^{\alpha/2}(\mathbb{R}^{n})}:= C_1.\label{ism1}
\end{eqnarray}
Moreover, using that the fractional Laplacian of a paraboloid is bounded, by \eqref{P} and \eqref{4.3bis}, it follows that
\begin{eqnarray}
I_2(x)&\leq& C(1+|x|^{n+\alpha})^{-1}\int_{\{y:\, |y|\leq\lambda\}}{u(x-y,t_1)\, dy}\nonumber\\
&\leq& C(1+|x|^{n+\alpha})^{-1}\int_{\{z:\, |x-z|\leq\lambda\}}{u(z,t_1)\, dz}\nonumber\\
&\leq& C(1+|x|^{n+\alpha})^{-1}\int_{\{z:\, |z|\leq2|x|\}}{u(z,t_1)P_{1}(z)\frac{1}{P_{1}(z)}\, dz}\nonumber\\
&\leq& C(1+|x|^{n+\alpha})^{-1}(1+(2|x|)^{n+\alpha})\int_{\{z:\, |z|\leq2|x|\}}{u(z,t_1)P_{1}(z)\, dz}\nonumber\\
&\leq& CM_{1}(1+|x|^{n+\alpha})^{-1}(1+(2|x|)^{n+\alpha})\nonumber\\
&\leq& C(n,\alpha)u(0,1+t_1):=C_2.\label{ism2}
\end{eqnarray}
By \eqref{ism},
\eqref{ism1} and \eqref{ism2}, we obtain \eqref{objetivo} and we conclude the proof.
\end{proof}
\begin{remark}
It would be interesting to prove that the pointwise behavior in Proposition \ref{point} and some comparison arguments provide an alternative proof to our
main result, namely, the representation of every solution as a convolution with the associated Poisson kernel.
\end{remark}

%%%%%%%%%%%%%%%%%%%%%%%%%%
%%%%%%%%%%%%%%%%%%%%%%%%%%
%%%%%%%%%%%%%%%%%%%%%%%%%%

%%%%%%%%%%%%%%%%%%%%%%%%%%
%%%%%%%%%%%%%%%%%%%%%%%%%%

\end{document}